\documentclass[10pt,a4paper]{amsart}
\usepackage{amssymb, amsmath, amsthm,amsfonts}
\usepackage[utf8]{inputenc}
\usepackage{aliascnt}
\usepackage[arrow, matrix, curve, graph]{xy}
\usepackage{longtable}
\usepackage{tikz}
\usetikzlibrary{matrix,arrows,positioning,backgrounds,shapes}
\usepackage{mathtools}

\theoremstyle{plain} 
\newtheorem{lemma}{Lemma}

\newaliascnt{prop}{lemma}
\newtheorem{prop}[prop]{Proposition}

\aliascntresetthe{prop}

\newaliascnt{coro}{lemma}
\newtheorem{coro}[coro]{Corollary}

\aliascntresetthe{coro}

\newaliascnt{theo}{lemma}
\newtheorem{theo}[theo]{Theorem}

\aliascntresetthe{theo}

\newtheorem*{theo*}{Theorem}

\theoremstyle{definition}

\newaliascnt{defn}{lemma}

\aliascntresetthe{defn}

\theoremstyle{remark}

\newaliascnt{rem}{lemma}

\aliascntresetthe{rem}

\begin{document}

\title[A theorem of Hertweck on $p$-adic conjugacy]{A theorem of Hertweck on $p$-adic conjugacy of $p$-torsion units in group rings}
\author{Leo Margolis}
\address{Departamento de matem\'aticas, Facultad de matem\'aticas, Universidad de Murcia, 30100 Murcia, Spain}
\email{leo.margolis@um.es}
\thanks{This "research" was supported by a Marie Curie Individual Fellowship from EU project 705112-ZC}
\date{\today}

\keywords{Unit Group, Group Ring, Zassenhaus Conjecture, $p$-adic Conjugacy}
\subjclass[2010]{16S34, 16U60, 20C11, 20C05}

\begin{abstract} A proof of a theorem of M.~Hertweck presented during a seminar in January 2013 in Stuttgart is given. The proof is based on a preprint given to me by Hertweck. 

Let $R$ be a commutative ring, $G$ a finite group, $N$ a normal $p$-subgroup of $G$ and denote by $RG$ the group ring of $G$ over $R$. It is shown that a torsion unit $u$ in $\mathbb{Z}G$ mapping to the identity under the natural homomorphism $\mathbb{Z}G \rightarrow \mathbb{Z}G/N$ is conjugate in the unit group of $\mathbb{Z}_pG$ to an element in $N$. Here $\mathbb{Z}_p$ denotes the $p$-adic integers. The result is achieved proving a result in the context of the so-called double action formalism for group rings over $p$-adic rings. This widely generalizes a theorem of Hertweck and a related theorem by Caicedo-Margolis-del R\'io and has consequences for the study of the Zassenhaus Conjecture for integral group rings.
\end{abstract}

\maketitle

\section{Introduction}
A long standing conjecture by Zassenhaus states that for $G$ a finite group and $u$ a torsion unit in the integral group ring $\mathbb{Z}G$ there exists a unit in the rational group algebra of $G$ conjugating $u$ onto an element of the form $\pm g$ for some $g \in G$. 

A main achievement in the study of this conjecture has been Weiss' proof for nilpotent groups \cite{Weiss91}. Weiss in fact showed that the unit conjugating $u$ onto $\pm g$ can even be taken in the $p$-adic group ring $\mathbb{Z}_pG$ \cite[Corollary on page 184]{Weiss91}. Though one can in general not hope that the conjugation will always take place in the $p$-adic group ring, the smallest non-abelian case $\mathbb{Z}S_3$ provides a well-known counterexample \cite[Example 3.4]{HertweckColloq}, in some instances this is true, as for Frobenius groups and units which map to the identity when factoring out the Frobenius kernel \cite{HertweckFrobenius}.

The idea that studying $p$-adic conjugacy of units in $\mathbb{Z}G$ can give insight into the Zassenhaus Conjecture, i.e. into conjugacy questions in the rational group algebra, was introduced by Hertweck in \cite{HertweckColloq} and was a main ingredient in the results of \cite{HertweckEdinb, CyclicByAbelian}. The result presented here greatly increases the knowledge on this topic and first applications of it to the Zassenhaus Conjecture will be published soon.

Before stating the result whose proof will occupy most of this note, we need to introduce the double-action module. Let $R$ be a commutative ring, $G$ a finite group and $C$ a cyclic group generated by an element $c$. For a torsion unit $u \in RG$ of order dividing the order of $c$ define the $R(G \times C)$-module ${}_G(RG)_u$. As $R$-module this is just $RG$ and the action of $G \times C$ is defined by
\[a \cdot (g, c^i) = g^{-1}au^i \ \ \text{for} \ \ a \in {}_G(RG)_u, \ g \in G, \ i \in \mathbb{Z}. \] 
This formalism allows to compare modules associated to different torsion units in $RG$. When speaking of a $p$-adic ring we will mean a complete discrete valuation ring of characteristic $0$ with residue class field of characteristic $p$. We can now state the results.
 
\begin{prop}\label{MainProp}
Let $R$ be a $p$-adic ring with quotient field $K$ and let $u$ be a torsion unit of $p$-power order in $RG$. Assume that $u$ is conjugate to an element $x$ of $G$ in the unit group of $KG$. Then $u$ is conjugate to $x$ in the unit group of $RG$ if and only if $_G(RG)_u$ is isomorphic to a direct summand of a direct sum of copies of $_G(RG)_x$. 
\end{prop}

Using previous results of Hertweck relying on the work of Weiss \cite{Weiss88} this implies the following. Here $\mathbb{Z}_{(G)}$ denotes the semilocalisation of $\mathbb{Z}$ at the primes dividing the order of the finite group $G$.

\begin{theo}\label{MainTheorem}
Let $G$ be a finite group and $N$ a normal $p$-subgroup of $G$. Then any torsion unit which maps to the identity under the natural homomorphism $\mathbb{Z}G \rightarrow \mathbb{Z}G/N$ is conjugate to an element of $N$ in the unit group of $\mathbb{Z}_pG$ and thus also in the unit group of $\mathbb{Z}_{(G)}G$. 
\end{theo}

This also emphasises the difference between studying torsion units or torsion subgroups in $\mathrm{V}(\mathbb{Z}G)$, since for subgroup it is known that such a result does not hold \cite[Example 4.1]{HertweckAnother}.

The following section introduces the basic concepts which connect torsion units and bimodules. After that we study the idempotents in $p$-adic group rings and their relation with partial augmentations and relative projectivity before proving Proposition~\ref{MainProp} and Theorem~\ref{MainTheorem} in the last section. Many of the results in the second and fourth paragraph are well known, see \cite{Thevenaz}, but we include proofs for completeness and to avoid too many assumptions on the ground ring $R$.\\

\textbf{Remark on this note:} This preprint is not intended for peer-reviewed publication, since this is not my own result. The results were publicly presented, with a sketch of the proof, by Hertweck in the Algebra Seminar of the University of Stuttgart on January 29th 2013, cf. the abstract of the talk \cite{HertweckSeminar}. A physical copy of a preprint proving the results was given to me by Hertweck in late March 2013. It is my understanding that Hertweck was planning a generalisation of these results, but I have no knowledge if this was achieved. Since Hertweck has not been known to be mathematically active since late 2013 and this result is of major interest in the study of torsion units of group rings, apart from having also a beautiful proof, I feel this way of publication as appropriate.

The proof presented here follows very closely the proof of Hertweck. The wording has been changed, a few details added, some notation changed, at places results have been made more general or more special, but all ideas and the basic structure remain.   

\section{Notation and basic facts}
Let $R$ be a commutative ring and $G$ a finite group. Denote by $\mathrm{V}(RG)$ the group of units of augmentation $1$ in $RG$, i.e. the units whose coefficients sum up to $1$. Let $\eta$ be a group homomorphism from a finite group $H$ to $\mathrm{V}(RG)$. Define a right $R(G \times H)$-module $ M = {}_1(RG)_\eta$ which is $RG$ as $R$-module and where the action of $G \times H$ is given by
\[m\cdot (g,h) = g^{-1}m\eta(h), \ \ m \in M, g \in G, h \in H. \]
Such a module is sometimes called a double-action module. Note that this module may also be viewed as an $RG^{\text{op}}\text{-} RH$-bimodule. Moreover the module $_G(RG)_u$ corresponds in this notation to $_1(RG)_\eta$ where $\eta: C \rightarrow \mathrm{V}(RG)$ maps $c$ to $u$.

The connection between double-action modules and conjugacy of subgroups in the unit group of $RG$ is given by the concept of $R$-equivalence of homomorphisms, cf. \cite{Weiss91}, \cite[Section (1.2)]{RoggenkampScottIso}. If $\sigma: H \rightarrow \mathrm{V}(RG)$ is another homomorphism then $\eta$ and $\sigma$ are called $R$-equivalent if there exists a unit $\mu$ in $RG$ such that $\eta(h) = \sigma(h)^\mu$ for all $h \in H$. Then $\eta$ and $\sigma$ are $R$-equivalent if and only if ${}_1(RG)_\eta \cong {}_1(RG)_\sigma$. More precisely such an isomorphism $\varphi$ between ${}_1(RG)_\eta$ and ${}_1(RG)_\sigma$ is given by right multiplication by the unit $\mu = \varphi(1)$ in $RG$ and this unit is the conjugating unit in the definition of $R$-equivalence.

Denote by $*$ the standard anti-involution on $RG$, i.e. the $R$-linear extension of $g^* = g^{-1}$ for $g \in G$ to $RG$. Then for $m \in M$ and $a \in RG$ we have $m \cdot a = a^*m$. Moreover if $e$ is an idempotent in $RG$ then the direct summand $RGe$ of $\operatorname{res}^{G \times H}_G M$ is projective.
The group homomorphism $\eta$ also yields a conjugation action of $H$ on $RG$. Denote the set of $H$-fixed points by
\[(RG)^H = \{m \in RG \ | \ m^{\eta(h)} = m \ \ \text{for all} \ \ h\in H \}. \]
Then $(RG)^H$ is a subring of $RG$ with the same identity. The $H$-fixed points can also be defined for $H$-invariant subspaces of $RG$.

We first collect some basic facts about the direct summands of $M$.
\begin{lemma}\label{IdemRG}
Let $R$ be a commutative ring, $H$ a finite group and $\eta: H \rightarrow \mathrm{V}(RG)$ a group homomorphism. Set $M = {}_1(RG)_{\eta}$.
\begin{itemize}
\item[a)] Any direct summand of $M$ is of the form $RGe$ for some idempotent $e$ in $(RG)^H$.
\item[b)] A direct indecomposable summand of $M$ is of the form $RGe$ for a primitive idempotent $e$ of $(RG)^H$.
\item[c)] Let $e$ be an idempotent in $(RG)^H$ and view $RGe$ as a direct summand of $M$. Then 
\[\operatorname{End}_{R(G \times H)}(RGe) \cong (eRGe)^H \cong e(RG)^He.\]	
 In particular $\operatorname{End}_{R(G\times H)}(RG) \cong (RG)^H$.
\end{itemize}
\end{lemma} 

\begin{proof}
\begin{itemize}
\item[a)] Let $U$ be a direct summand of $M$ and $V$ a submodule of $M$ such that $M = U \oplus V$. Let $1 = e + f$ with $e \in U$ and $f \in V$.
Then $e$ and $f$ are orthogonal idempotents since
\[e = 1\cdot e^* = e \cdot e^* + f \cdot e^* = e^2 + ef \]
where $e^2 \in U$ and $ef \in V$. So from $RG = RGe \oplus RGf$ we conclude $U = RGe$. Moreover for $h \in H$ we have
\[\eta(h)e + \eta(h)f = \eta(h) = 1 \cdot h = e \cdot h + f \cdot h. \]
So $M = U \oplus V$ implies $\eta(h)e = e \cdot h = e\eta(h)$ and hence $e \in (RG)^H$.
\item[b)] This follows from a).
\item[c)] The isomorphism $\operatorname{End}_{RG}(RGe) \cong eRGe$ is given by mapping $\varphi$ to $e\varphi$. Now let $\varphi \in \operatorname{End}_{R(G\times H)}(RGe)$ and $h \in H$. Then from
\begin{align*}
(e\varphi)\eta(h) =& (e\varphi) \cdot h = (e \cdot h)\varphi = (e\eta(h))\varphi = (\eta(h)e)\varphi = (e \cdot \eta(h)^*)\varphi \\
 =& (e\varphi) \cdot \eta(h)^* = \eta(h)(e\varphi) 
\end{align*}
we get that $e\varphi \in (eRGe)^H$.
\end{itemize}
\end{proof}

\begin{lemma}\label{IdemConj}
Let $R$ be a complete noetherian local ring, e.g. a complete discrete valuation ring. 
\begin{itemize}
\item[a)] Let $e$ be an idempotent in $(RG)^H$. Then $e$ is primitive if and only if $e(RG)^He$ is local.
\item[b)] For idempotents $e$ and $f$ in $(RG)^H$ let $\varphi: RGe \rightarrow RGf$ be an isomorphism of $R(G \times H)$-modules. Then $\varphi$ is given as the right multiplication with a unit $\mu$ in $(RG)^H$ and $e^\mu = f$.
\item[c)] There is a bijection between the isomorphism classes of indecomposable direct summands of $M$ and conjugacy classes of primitive idempotents in $(RG)^H$. This conjugacy action is understood to be taking place in the units of $(RG)^H$.
\end{itemize}
\end{lemma} 
\begin{proof}
\begin{itemize}
\item[a)] By assumption a lifting theorem for idempotents is available \cite[(30.4) Theorem]{CR} implying the claim by Lemma~\ref{IdemRG}.
\item[b)] We have 
\[RGe \oplus RG(1-e) = RG = RGf \oplus RG(f-1).\]
Since the Krull-Schmidt Theorem \cite[(30.6) Theorem]{CR} holds we obtain an isomorphism between $RG(1-e)$ and $RG(1-f)$ and the claim follows from \cite[Exercises 14 in §6]{CR}. 
\item[c)] This is a direct consequence of b) and Lemma~\ref{IdemRG}.
\end{itemize}
\end{proof}

A fundamental notion when studying conjugacy of units are the so called partial augmentations. Let $a = \sum_{g \in G} r_gg$ be an element in $RG$ and denote by $g^G$ the conjugacy class of an element $g$ in $G$. Then $\varepsilon_{g^G}(a) = \sum_{h \in g^G} r_h$ is called the partial augmentation of $a$ at $g^G$. Then $\varepsilon_{g^G}$ is a trace function, i.e. $\varepsilon_{g^G}(ab) = \varepsilon_{g^G}(ba)$ for any $a,b \in RG$ \cite[Lemma (7.2)]{SehgalBook}. Note that $\varepsilon_{g^G}(a)$ is frequently denoted $\varepsilon_g(a)$, but we prefer the first notation since we will be distinguishing between conjugacy in $G$ and subgroups of $G$.

We will also need the character associated to the double-action module $M = {}_1(RG)_{\eta}$. Since $M$ is free as $R$-module we can associate to it a character $\theta$. Fixing $G$ as a basis of $M$ an easy calculation \cite[Lemma (38.12)]{SehgalBook} then shows that for any $(g,h) \in G \times H$ we have
\[\theta((g,h)) = |C_G(g)|\varepsilon_{g^G}(\eta(h)). \]
For an idempotent $e$ in $(RG)^H$ let $\theta_e$ be the character associated to the direct summand $RGe$ of $M$. Since $e\eta(H)$ annihilates $RG(1-e)$ while acting as $\eta(H)$ on $RGe$ we obtain
\begin{align}\label{theta_egh}
\theta_e((g,h)) = |C_G(g)|\varepsilon_{g^G}(e\eta(h)).
\end{align}
In particular for $H = 1$ we obtain the characters of the right projective $RG$-lattices:
\begin{align}\label{theta_e}
\theta_e(g) = |C_G(g)|\varepsilon_{g^G}(e). 
\end{align}

\section{Idempotents and partial augmentations}
For a commutative ring $R$ define $[RG,RG]$ to be the additive commutator of $RG$, i.e. the $R$-linear span of all elements of the form $[a,b] = ab - ba$ for $a,b \in RG$. Then $[RG,RG]$ consists exactly of those elements $a$ of $RG$ satisfying $\varepsilon_{g^G}(a) = 0$ for all $g \in G$ \cite[Lemma (7.2)]{SehgalBook}. Note that for $R$ of characteristic $0$ an element $a \in [RG,RG]$ satisfies $\chi(a) = 0$ for any ordinary character of $G$. Moreover for $a \in RG$ denote by $\text{supp}(a)$ the support of $a$, i.e. the set of elements of $G$ having non-vanishing coefficient in $a$.

We first consider linear independence modulo $[RG,RG]$ of idempotents in $RG$.

\begin{lemma}\label{LinInd}
Let $(K,R,k)$ be a $p$-modular system and let $e_1,...,e_n$ be representatives of the conjugacy classes of primitive idempotents in $RG$. If $\lambda_1,...,\lambda_n$ are elements in $R$ such that $\lambda_1e_1 + ... + \lambda_ne_n \in [RG,RG]$ then $\lambda_1 = ... = \lambda_n = 0$.
\end{lemma}

\begin{proof}
Let $\chi_1,...,\chi_m$ be the irreducible $K$-characters of $G$ with corresponding primitive idempotents $f_1,...,f_m$ in $KG$. Let $e = (e_{i,j})$ be the homomorphism between the Grothendieck groups of projective $RG$-modules and $KG$-modules, as defined in \cite[§ 18]{CR}. Thus $e$ assigns to a projective $RG$-module, with respect to the basis $RGe_1,...,RGe_n$, its $KG$ composition factors with respect to the basis $KGf_1,...,KGf_m$. So for each $i$ we have $K \otimes_R RGe_i \cong \sum_{j=1}^m e_{j,i} KGf_j.$ Now $\chi_\ell(f_j) = 0$ for $j \neq \ell$ and so 
\[\chi_\ell(e_i) = \chi_\ell\left(\sum_{j=1}^m e_{j,i} f_j\right) = e_{\ell, j}\chi_\ell(f_\ell).\]

From the assumption that $\sum_{i=1}^m \lambda_i e_i \in RG$ we know $\chi_\ell(\sum_{i=1}^m \lambda_i e_i) = 0$ for any $\chi_\ell$ and so
\[0 = \begin{pmatrix} \chi_1(\sum_{i=1}^n \lambda_ie_i)/\chi_1(f_1) \\ \vdots \\ \chi_m(\sum_{i=1}^n \lambda_ie_i)/\chi_m(f_m) \end{pmatrix} = e \begin{pmatrix}
 \lambda_1 \\ \vdots \\ \lambda_n \end{pmatrix}. \]
But by \cite[(21.20) Theorem]{CR} the map $e$ is a split injection and thus the above equality implies $\lambda_1 = ... = \lambda_n = 0$.   
\end{proof}

We next proof a generalization of an observation of Külshammer \cite[Proposition 3]{Kuelshammer} using a result of Swan. For a Dedekind ring $R$ of characteristic $0$ an element $g \in G$ is called $R$-singular, if the order of $g$ is not a unit in $R$. For a prime $p$ denote by $R[G_{p'}]$ the $R$-linear span of $p$-regular elements in $G$, i.e. those elements whose order is not divisible by $p$.

\begin{prop}\label{SuppIdem}
Let $R$ be a Dedekind ring of characteristic $0$ and $e$ an idempotent in $RG$. Then $\varepsilon_{g^G}(e) = 0$ for any $R$-singular element $g \in G$. In particular for $p$ a prime not invertible $R$ we have $e \in R[G_{p'}] + [RG,RG]$.
\end{prop}
\begin{proof}
Let $\theta_e$ be the character associated to the projective $RG$-module $RGe$. By \eqref{theta_e} we know $\theta_e(g) = |C_G(g)|\varepsilon_{g^G}(e)$ for $g\in G$. Since $\theta_e$ vanishes on $R$-singular elements \cite[(32.15) Theorem]{CR} the result follows. 
\end{proof}

We need to establish one more connection between idempotents and partial augmentations and the following elementary observation will be quite useful.
\begin{lemma}\label{PartAugGOrH}
Let $x \in G$ be a non-trivial $p$-element. Set $H = C_G(x)$ and let $a \in R[H_{p'}]$. Then $\varepsilon_{h^G}(ax) = \varepsilon_{h^H}(ax)$ for any $h \in H$.

Moreover let $c \in [RH, RH]$ and assume that for $g \in G$ 
\[\varepsilon_{g^G}((a+c)x) = \left\{\begin{array}{ll} 1, \ \text{if} \ g \in x^G \\ 0, \ \text{otherwise.} \end{array} \right. \]
Then for $h \in H$ we have
\[\varepsilon_{h^H}(a+c) = \left\{\begin{array}{ll} 1, \ \text{if} \ h = 1 \\ 0, \ \text{otherwise.} \end{array} \right. \]
\end{lemma}
\begin{proof}
If $h^G \cap \text{supp}(ax) = \emptyset$ then clearly $\varepsilon_{h^G}(ax) = \varepsilon_{h^H}(ax) = 0$. If on the other hand $h_1x$ and $h_2x$ are elements in $\text{supp}(ax)$ which are conjugate in $G$, where $h_1, h_2 \in H_{p'}$, then they have the same $p$-part, namely $x$. Thus conjugation between $h_1x$ and $h_2x$ takes place already in $H = C_G(x)$ and so $\varepsilon_{(h_1x)^G}(ax) = \varepsilon_{(h_1x)^H}(ax)$.

For the second statement observe that since $x$ lies in the centre of $H$ multiplication by $x$ is permuting the conjugacy classes in $H$ and so we have $cx \in [RH,RH]$. Thus $\varepsilon_{g^G}((a + c)x) =  \varepsilon_{g^G}(ax)$ and so we can assume $c = 0$. So by the first statement our assumption implies
\[\varepsilon_{h^H}(ax) = \left\{\begin{array}{ll} 1, \ \text{if} \ h \in x^H \\ 0, \ \text{otherwise} \end{array} \right. \]
and this implies the second statement of the lemma.
\end{proof}

Let $R$ be a $p$-adic ring and $e$ a primitive idempotent in $RG$. Then the multiplicity of $e$ is the number of conjugates of $e$ appearing in an orthogonal primitive decomposition of $1$ in $RG$. Here conjugation is understood to take place in the unit group of $RG$.

\begin{coro}\label{MultIdem}
Let $R$ be a $p$-adic ring, $x \in G$ a non-trivial $p$-element and set $H = C_G(x)$. Let $e_1,...,e_n$ be representatives of conjugacy classes of primitive idempotents in $RH$ with multiplicity $m_1,...,m_n$ respectively. Let $k_1,...,k_n$ be integers such that for any $g \in G$ we have
\[\varepsilon_{g^G}\left( \sum_{i=1}^n k_ie_ix \right) = \left\{\begin{array}{ll} 1, \ \text{if} \ g \in x^G \\ 0, \ \text{otherwise} \end{array} \right. \]
Then $k_i = m_i$ for all $1 \leq i \leq n$.
\end{coro}
\begin{proof}
By Proposition~\ref{SuppIdem} we know $e_i \in R[H_{p'}] + [RH, RH]$ for any $1 \leq i \leq n$. So by our assumption and Lemma~\ref{PartAugGOrH} we obtain
\[\varepsilon_{h^H}\left( \sum_{i=1}^n k_ie_i \right) = \left\{\begin{array}{ll} 1, \ \text{if} \ h = 1 \\ 0, \ \text{otherwise} \end{array} \right. \]
Note that if $e$ and $f$ are conjugate elements in $RH$ then $\varepsilon_{h^H}(e) = \varepsilon_{h^H}(f)$ for any $h \in H$. So
\[\varepsilon_{h^H}\left( \sum_{i=1}^n m_ie_i \right) = \varepsilon_{h^H}(1) = \left\{\begin{array}{ll} 1, \ \text{if} \ h = 1 \\ 0, \ \text{otherwise} \end{array} \right. \]
Hence $\sum\limits_{i=1}^n (m_i - k_i)e_i \in [RH, RH]$ and thus $m_i = k_i$ for any $1 \leq i \leq n$ by Lemma~\ref{LinInd}.
\end{proof}

\section{Relative Projectivity and Idempotents}
This section is devoted to the proof of Lemma~\ref{RelProjEquiv} which, in the words of Hertweck, "seems to be more or less known" and information on it can be found e.g. in Thevenaz's book \cite{Thevenaz}. Part d) of the lemma seems however to be less well known. We will first introduce the necessary notation. Let $R$ be a commutative ring and $K$ and $H$ subgroups of $G$ such that $K \leq H$. Denote by $[K\setminus H]$ a transversal of the right cosets of $K$ in $H$. Then $(RG)^H \subseteq (RG)^K$ and we can define the relative trace map
\begin{align*}
&\operatorname{Tr}_K^H : (RG)^K \rightarrow (RG)^H \\
&\operatorname{Tr}_K^H(a) = \sum_{g \in [K\setminus H]} a^g.
\end{align*}
Abusing notation we will write the image of $\operatorname{Tr}_K^H$ as $\operatorname{Tr}_K^H(RG)$. We will need the following result which follows by lifting idempotents.

\begin{lemma}\label{IdemDecAT}
Let $R$ be a $p$-adic ring, $x \in G$ a non-trivial $p$-element and set $H = C_G(x)$. Furthermore set
\[A = (RG)^{\langle x \rangle} \ \ \text{and} \ \ T = \operatorname{Tr}_{\langle x^p \rangle}^{\langle x \rangle}(RG). \]

For $e_1,...,e_r$ a set of representatives of the conjugacy classes of primitive idempotents in $RH$ there exists a set $f_1,...,f_r,...,f_s$ of representatives of conjugacy classes of primitive idempotents in $A$ where $r \leq s$ and $f_i \in T$ for $i > r$. Moreover for $i \leq r$ the difference $e_i - f_i = \varepsilon_i$ is an idempotent of $T$ which is orthogonal to $f_i$. 
\end{lemma}
\begin{proof}
$T$ is an ideal of $A$ and $A = RH + T$. Moreover for $a \in RH$ we have $\operatorname{Tr}_{\langle x^p \rangle}^{\langle x \rangle}(a) = pa$, so $RH \cap T = pRH$. Thus we obtain
\[A/T = (RH+T)/T \cong RH/(RH \cap T) \cong RH/pRH \cong (R/pR)H. \]
Set $\bar{A} = A/T$. Let $1 = \sum_{i \in I} e_i$ be a primitive orthogonal idempotent decomposition in $RH$ for a certain index set $I$. By the assumption on $R$ a lifting theorem  is available \cite[(3.1) Theorem]{Thevenaz} showing that $1 = \sum_{i \in I} \bar{e}_i$ is a primitive orthogonal idempotent decomposition in $\bar{A} = (R/pR)H$. For each $i \in I$ let $e_i = \sum_{j \in J_i} f_{i,j}$ be a primitive orthogonal idempotent decomposition in $A$ for certain $J_i$. Then $1 = \sum_{i \in I, j \in J_i} f_{i,j}$ is a primitive orthogonal idempotent decomposition in $A$. Every primitive idempotent in $A$ is conjugate to one of the $f_{i,j}$. 

Fix some $i \in I$. Since $\bar{e}_i = \sum_{j \in J_i} \bar{f}_{i,j}$ is primitive all $\bar{f}_{i,j}$ are $0$ except one, say $\bar{f}_{i,j_0}$. In particular $f_{i,j}$ lies in $T$ for $j \neq j_0$. Set $\varepsilon = \sum_{j \in J_i, j \neq j_0} f_{i,j}$, so $\varepsilon$ is an idempotent in $T$ orthogonal to $f_{i,j_0}$. Then $e_i - f_{i,j_0} = \varepsilon$. Finally if $f$ is another primitive idempotent in $A$ such that $e_i - f$ lies in $T$ then $f$ and $f_{i,j_0}$ are conjugate by \cite[(3.2) Theorem (d)]{Thevenaz}.
\end{proof}

\textbf{Remark:} In the proof we used \cite[Theorems 3.1, 3.2]{Thevenaz} where it is assumed that the residue class field of $R$ is algebraically closed. This is however not needed in the proof of these theorems, so we can apply them in our setting.\\ 

We will also need the following variation of Green's indecomposability theorem. Denote by $Z(G)$ the centre of $G$.
\begin{lemma}\label{GreenInd}
Let $R$ be a $p$-adic ring. Let $H$ be a subgroup of $G$ of $p$-power index such that $G = HZ(G)$. If $V$ is an indecomposable $RH$-module then $\operatorname{ind}_H^GV$ is an indecomposable $RG$-module.
\end{lemma}
\begin{proof}
Arguing by induction we can assume $[G:H] = p$. Set $W = \operatorname{ind}_H^G V$ and denote $B = \operatorname{End}_{RH}(V)$ and $A = \operatorname{End}_{RG}(W)$. We will imitate the arguments of \cite[Proposition 12.10]{Dade}. We view $B$ as an $R$-subalgebra of $A$. Let $g \in Z(G)$ be a $p$-element such that $g^p \in H$. Denote by $\bullet g$ the endomorphism of $W$ mapping $w$ to $wg$. Then $\bullet g$ is a central element of $A$ and by Frobenius reciprocity we have 
\[A \cong B \oplus (\bullet g) B \oplus ... \oplus  (\bullet g)^{p-1} B  \]
as $R$-modules. In particular $A$ is generated by $B$ and $\bullet g$. Denote by $J(B)$ and $J(A)$ the radical of $B$ and $A$ respectively. Denoting by $\overline{\bullet g}$ the image of $\bullet g$ in $\overline{A} = A/J(B)A$ we find that $(\overline{\bullet g} - 1)$ generates a nilpotent ideal in $\overline{A}$, since the latter is a ring of characteristic $p$ and $\overline{\bullet g}$ of $p$-power order. Setting $\overline{B} = B/(B \cap J(B)A)$ we thus get an isomorphism
\[\overline{A}/(\overline{\bullet g} -1)\overline{A} \cong \overline{B}/(\overline{B} \cap (\overline{\bullet g} -1)\overline{A}). \]
By the indecomposability of $V$ we know that $B/J(B)$ is a division ring and thus so is its quotient $\overline{B}/(\overline{B} \cap (\overline{\bullet g} -1)\overline{A})$. Hence so is $A/J(A) \cong \overline{A}/(\overline{\bullet g} -1)\overline{A}$. So $A$ is local and $W$ is indecomposable.  
\end{proof}

We are now ready to prove the main lemma of this paragraph connecting all the concepts developed so far. For a summand $V$ of the double-action module ${}_G(RG)_u$ associated to a torsion unit $u$ in $RG$ we will write sometimes ${}_G(V)_u$.
\begin{lemma}\label{RelProjEquiv}
Let $R$ be a $p$-adic ring, $x \in G$ a non-trivial $p$-element and $e$ in $(RG)^{\langle x \rangle}$ a primitive idempotent. View ${}_G(RG)_x$ as an $R(G \times C)$-module where $C$ is a cyclic group of the same order as $x$. The following are equivalent.
\begin{itemize}
\item[a)] ${}_G(RGe)_x$ is projective relative to $G \times C^p$.
\item[b)] There exists a primitive idempotent $\nu \in (RG)^{\langle x^p \rangle}$ such that
\[e = \operatorname{Tr}_{\langle x^p \rangle}^{\langle x \rangle}(\nu) = \nu + \nu^x + ... + \nu^{x^{p-1}} \]
is an orthogonal decomposition of $e$.
\item[c)] $e \in  \operatorname{Tr}_{\langle x^p \rangle}^{\langle x \rangle}(RG)$.
\item[d)] $ex \in [RG, RG]$. 
\end{itemize}
\end{lemma}
\begin{proof}
Set $M = {}_G(RGe)_x$ and let $C = \langle c \rangle$.

a) $\Rightarrow$ b): By \cite[(19.2) Theorem (iv)]{CR} $M$ is a direct summand of the module	 $\operatorname{ind}_{G \times C^p}^{G \times C} \operatorname{res}_{G \times C^p}^{G\times C} M$. So, since $M$ is indecomposable, there exists a direct summand $V$ of $\operatorname{res}_{G \times C^p}^{G\times C} M$ such that $M$ is a direct summand of $\operatorname{ind}_{G \times C^p}^{G\times C}V$. But $\operatorname{ind}_{G \times C^p}^{G\times C}V$ is indecomposable by Lemma~\ref{GreenInd} and so
\[M \cong \operatorname{ind}_{G \times C^p}^{G\times C}V = (V \otimes 1) + (V \otimes c) + ... + (V \otimes c^{p-1}). \]
Hence as $R$-modules we have
\[M = V_0 \oplus V_1 \oplus ... \oplus V_{p-1} \]
where each $V_i$ is a direct summand of $\operatorname{res}_{G \times C^p}^{G\times C} M$ isomorphic to $V$ and $V_ix = V_i \cdot c = V_{i+1}$ for all $0\leq i \leq p-1$. Here indices are understood to be taken modulo $p$. i.e. $V_p = V_0$. Let
\[e = \nu_0 + \nu_1 +... + \nu_{p-1} \]
with $\nu_i \in V_i$. Then the $\nu_i$ are primitive and pairwise orthogonal idempotents in $(RG)^{\langle x^p \rangle}$ and $V_i = RG\nu_i$. Since $e$ is centralized by $x$ we obtain
\[e = e^x = \nu_0^x + \nu_1^x +... + \nu_{p-1}^x = \nu_{p-1}^x + \nu_0^x +... + \nu_{p-2}^x \]
where $\nu_i^x \in V_{i+1}$. So $\nu_i^x = \nu_{i+1}$ for $0 \leq i \leq p-1$ and b) holds with $\nu = \nu_0$.

b) $\Rightarrow$ c) is clear.

c) $\Rightarrow$ a): Let $a \in (RG)^{\langle x^p \rangle}$ such that $e = \operatorname{Tr}_{\langle x^p \rangle}^{\langle x \rangle}(a)$. Since $e$ is centralized by $x$ we have
\[e = e \operatorname{Tr}_{\langle x^p \rangle}^{\langle x \rangle}(a) e = \operatorname{Tr}_{\langle x^p \rangle}^{\langle x \rangle}(eae). \]
For an element $m \in M$ denote by $\bullet m$ the map from $M$ to $M$ given by right multiplication with $m$. Then $\bullet eae$ is an endomorphism of $\operatorname{res}_{G \times C^p}^{G\times C} M$. Moreover $\bullet e$ is the identity map on $M$ and 
\[\bullet e  = \bullet \operatorname{Tr}_{\langle x^p \rangle}^{\langle x \rangle}(eae) = \sum_{i =0}^{p-1} \bullet(eae)^{x^i} = \sum_{i=0}^{p-1} c^{-i}(\bullet eae)c^i. \]
So $M$ is projective relative to $G \times C^p$ by Higman's criterion \cite[(19.2) Theorem (iii)]{CR}.

b) $\Rightarrow$ d): By assumption as $R$-modules we have $RGe = \oplus_{i = 0}^{p-1} RG\nu^{x^i}$. So the character $\theta_e$ corresponding to the $R(G \times C)$-module $M$ is induced from a character of an $R(G \times C^p)$-module, namely ${}_G(RG \nu)_{x^p}$. Since for any $g \in G$ the element $(g,c) \in G \times C$ is not conjugate to any element in $G \times C^p$ this implies $\theta_e((g,c)) = 0$. But by \eqref{theta_egh} $\theta_e((g,c)) = |C_G(g)|\varepsilon_{g^G}(ex)$. So $\varepsilon_{g^G}(ex) = 0$ for all $g \in G$ and this is equivalent to $ex \in [RG,RG]$.

d) $\Rightarrow$ c): By assumption we have $\varepsilon_{g^G}(ex) = 0$ for all $g \in G$. We will argue by contradiction, so assume $e \notin \operatorname{Tr}_{\langle x^p \rangle}^{\langle x \rangle}(RG)$. Setting $H = C_G(x)$ we know by Lemma~\ref{IdemDecAT} that there exist a primitive idempotent $e_1$ in $RH$ and an idempotent $\varepsilon$ in $\operatorname{Tr}_{\langle x^p \rangle}^{\langle x \rangle}(RG)$ such that $e = e_1 - \varepsilon$. Moreover $\varepsilon$ is the sum of primitive idempotents in $\operatorname{Tr}_{\langle x^p \rangle}^{\langle x \rangle}(RG)$. So by applying the implication c) $\Rightarrow$ b) $\Rightarrow$ d) to each primitive summand of $\varepsilon$ we obtain $\varepsilon_{g^G}(\varepsilon x) = 0$ for all $g \in G$. Moreover there exists some $h \in H$ such that $\varepsilon_{h^H}(e_1) \neq 0$. Since $x$ in central in $H$ this implies $\varepsilon_{(hx)^H}(e_1x) \neq 0$ and so $\varepsilon_{(hx)^G}(e_1x) \neq 0$ by Lemma~\ref{PartAugGOrH}. Altogether we get
\[\varepsilon_{(hx)^G}(e x) = \varepsilon_{(hx)^G}(e_1 x) - \varepsilon_{(hx)^G}(\varepsilon x) \neq 0, \]
contradicting our assumption.
\end{proof}

We will record one more fact following from the proof of a) $\Rightarrow$ b) in the preceding lemma.
\begin{lemma}\label{lastlemma}
Let $R$ be a $p$-adic ring, $x \in G$ a non-trivial $p$-element and $e \in \operatorname{Tr}_{C^p}^{C}(RG)$ a primitive idempotent. If $V$ is an indecomposable direct summand of $\operatorname{res}_{G \times C^p}^{G \times C}{}_G(RGe)_x$ then ${}_G(RGe)_x \cong \operatorname{ind}_{G \times C^p}^{G \times C} V$.
\end{lemma}

\section{Proof of Proposition~\ref{MainProp} and Theorem~\ref{MainTheorem} } 
\textit{Proof of Proposition~\ref{MainProp}:} If $u$ and $x$ are conjugate in $RG$ then the $R(G \times C)$-modules ${}_G(RG)_x$ and ${}_G(RG)_u$ are isomorphic. So assume that ${}_G(RG)_u$ is isomorphic to a direct summmand of a direct sum of copies of ${}_G(RG)_x$. Set $H = C_G(x)$ and $T = \operatorname{Tr}_{\langle x^p \rangle}^{\langle x \rangle}(RG)$. 

Let $f_1,...,f_r,...,f_s$ be conjugacy classes of primitive idempotents in $(RG)^{\langle x \rangle}$ as described in Lemma~\ref{IdemDecAT}. So $f_i \in T$ for $i > r$ and for $i \leq r$ we have an idempotent $e_i$ in $RH$ and an idempotent $\varepsilon_i$ in $T$ orthogonal to $f_i$ such that $e_i = f_i + \varepsilon_i$. Moreover the $e_1,...,e_r$ are representatives of the conjugacy classes of primitive idempotents in $RH$. 

The indecomposable summands of ${}_G(RG)_x$ are of the form $RGf_i$ by Lemma~\ref{IdemRG}. So by our assumption that ${}_G(RG)_u$ is a direct summand of a direct sum of copies of ${}_G(RG)_x$ there exist non-negative integers $k_1,...,k_s$ such that
\begin{align}\label{6.2}
{}_G(RG)_u \cong \oplus_{i=1}^s {{}_G(RGf_i)_x}^{k_i}.
\end{align}
Here all summands on the right hand side are indecomposable and different indices correspond to non-isomorphic modules by Lemma~\ref{IdemConj}.

Let $j_1,...,j_t$ be representatives of the conjugacy classes of primitive idempotents in $(RG)^{\langle u \rangle}$ and let $\ell_i$ be the multiplicity of $j_i$ in ${}_G(RG)_u$ for $1 \leq i \leq t$. So
\begin{align}\label{6.3}
{}_G(RG)_u \cong \oplus_{i=1}^t {{}_G(RGj_i)_u}^{\ell_i}.
\end{align}
Here also the summands on the right hand side are indecomposable and different indices correspond to non-isomorphic modules.

Let $I = \{i \ | \ 1\leq i \leq s, k_i \neq 0  \}$. Then comparing \eqref{6.2} and \eqref{6.3} we find by the Krull-Schmidt Theorem \cite[(30.6) Theorem]{CR} that there is a bijection $\sigma: I \rightarrow \{1,...,t\}$ such that
\begin{align}\label{6.3.5}
{}_G(RGf_i)_x \cong {}_G(RGj_{\sigma(i)})_u \ \ \text{and} \ \ k_i = \ell_{\sigma(i)}
\end{align}
for every $i \in I$. By Lemma~\ref{IdemConj} an isomorphism $\varphi_i: {}_G(RGf_i)_x \rightarrow {}_G(RGj_{\sigma(i)})_u$ is given by right multiplication with a unit $\mu_i$ in $RG$ and $f_i^{\mu_i} = j_{\sigma(i)}$. Since $\varphi_i$ is an isomorphism of $R(G \times C)$-modules we get
\[(f_ix)\mu_i = \varphi_i(f_i \cdot c) = \varphi_i(f_i) \cdot c = (f_i\mu_i)u = (\mu_i j_{\sigma(i)})u, \]
so $(f_ix)^{\mu_i} = j_{\sigma(i)}u$. 

Note that if $j$ and $j'$ are conjugate idempotents in $(RG)^{\langle u \rangle}$ we have $\varepsilon_{g^G}(ju) = \varepsilon_{g^G}(j'u)$ for any $g \in G$. So for $g \in G$ we have from \eqref{6.3.5} and \eqref{6.3}
\begin{align*}
\varepsilon_{g^G}(u) =& \varepsilon_{g^G}\left(\sum_{i=1}^t \ell_i j_i u\right) = \varepsilon_{g^G}\left(\sum_{i \in I} \ell_{\sigma(i)} j_{\sigma(i)} u \right) = \varepsilon_{g^G}\left(\sum_{i \in I} k_i (f_i x)^{\mu_i} \right) \\
 =& \varepsilon_{g^G}\left(\sum_{i \in I} k_i (f_i x) \right).
\end{align*} 
For $i > r$ we have $f_i \in T$, so c) in Lemma~\ref{RelProjEquiv} is satisfied, and by d) in Lemma~\ref{RelProjEquiv} we have $\varepsilon_{g^G}(f_ix) = 0$ for all $g \in G$. For the same reason for $i \leq r$ we have $\varepsilon_{g^G}(f_ix) = \varepsilon_{g^G}(e_ix - \varepsilon_ix) = \varepsilon_{g^G}(e_ix)$ for any $g \in G$. So for $g \in G$ we obtain
\begin{align}\label{6.3.6}
\varepsilon_{g^G}(u) = \varepsilon_{g^G} \left(\sum_{i \in I, \ i \leq r} k_ie_ix \right). 
\end{align}

Let $m_i$ be the multiplicity of $e_i$ in $RH$ for $1 \leq i \leq r$. Since the $e_i$ lie in $(RG)^{\langle x \rangle}$ the $R(G \times C)$-module $RGe_i$ is a direct summand of ${}_G(RG)_x$ by Lemma~\ref{IdemRG} and from $RG = RG \otimes_{RH} RH$  we obtain
\begin{align}\label{6.1}
{}_G(RG)_x \cong \oplus_{i = 1}^r {{}_G(RGe_i)_x}^{m_i}.
\end{align} 	
Note that the summands here are not necessarily indecomposable, since the $e_i$ might be not primitive in $(RG)^{\langle x \rangle}$. So for $g \in G$ we get
\begin{align}\label{6.6.7}
\varepsilon_{g^G} \left(\sum_{i=1}^r m_ie_ix \right) = \varepsilon_{g^G}(x) = \left\{\begin{array}{ll} 1, \ \text{if} \ g \in x^G \\ 0, \ \text{otherwise} \end{array} \right. . 
\end{align}
Since $u$ and $x$ are conjugate in $KG$ we have $\varepsilon_{g^G}(u) = \varepsilon_{g^G}(x)$ for all $g \in G$ by \cite[Lemma 2.5]{HertweckColloq}. So comparing \eqref{6.3.6} and \eqref{6.6.7} we get from Corollary~\ref{MultIdem} that $\{1,...,r\}$ is a subset of $I$ and $k_i = m_i$ for all $1 \leq i \leq r$.

Set 
\[V = \oplus_{i=1}^r {{}_G(RG\varepsilon_i)_x}^{m_i}  \ \ \text{and} \ \ W = \oplus_{i = r+1}^s {{}_G(RGf_i)_x}^{k_i}. \]
Since for any $1 \leq i \leq r$ the idempotents $f_i$ and $\varepsilon_i$ are orthogonal we have ${}_G(RGe_i)_x = {}_G(RGf_i)_x \oplus {}_G(RG\varepsilon_i)_x$. So from \eqref{6.2} and \eqref{6.1} we get
\begin{align}\label{6.4}
{}_G(RG)_u \oplus V &\cong \left(\oplus_{i = 1}^r {{}_G(RGf_i)_x}^{m_i} \right) \oplus V \oplus \left(\oplus_{i = r+1}^s {{}_G(RGf_i)_x}^{k_i} \right) \\
 &\cong {}_G(RG)_x \oplus W \nonumber. 
\end{align}
The units $u^p$ and $x^p$ are conjugate in $KG$ and if we restrict both sides in \eqref{6.2} to $G \times C^p$ we see that $u^p$ satisfies the assumptions of the Proposition. So arguing by induction on the order of $u$ we can assume that $u^p$ and $x^p$ are conjugate in the units of $RG$ or, in other words,
\[\operatorname{res}_{G \times C^p}^{G \times C} {}_G(RG)_u \cong \operatorname{res}_{G \times C^p}^{G \times C} {}_G(RG)_x. \]
So restricting both sides of \eqref{6.4} to $G \times C^p$ by Krull-Schmidt we get
\[\operatorname{res}_{G \times C^p}^{G \times C} V \cong \operatorname{res}_{G \times C^p}^{G \times C} W. \]
Now since $\varepsilon_i \in T$ and $f_i \in T$ for $i > r$ the isomorphism type of $V$ and $W$ is determined by their restrictions to $G \times C^p$ by Lemma~\ref{lastlemma}. So $V \cong W$ and cancelling $V$ and $W$ in \eqref{6.4} the Krull-Schmidt Theorem implies
\[{}_G(RG)_u \cong {}_G(RG)_x.\]
So $u$ and $x$ are conjugate in $RG$.\\

\textit{Proof of Theorem~\ref{MainTheorem}:}  Let $u$ be a torsion unit mapping to the identity under the natural homomorphism $\mathbb{Z}G \rightarrow \mathbb{Z}G/N$. Then $u$ is conjugate to an element $x \in N$ by a unit in $\mathbb{Q}G$ by \cite[Proposition 4.2]{HertweckColloq}. Note that this relies on the result of Weiss on permutation modules over $p$-adic rings \cite{Weiss88}. Then by \cite[Corollary 3.2]{HertweckEdinb} the $\mathbb{Z}_p(G \times C)$-module ${}_G(\mathbb{Z}_p G)_u$ is isomorphic to a direct summand of a direct sum of copies of ${}_G(\mathbb{Z}_p G)_x$. Hence $u$ and $x$ are conjugate in $\mathbb{Z}_pG$ by Proposition~\ref{MainProp}. Finally that $u$ and $x$ are then already conjugate by a unit in $\mathbb{Z}_{(G)}G$ can be seen from the arguments in \cite[§ 2]{HertweckFrobenius}. \\

\textbf{Acknowledgement:} I am thankful to \'A. del R\'io for helpful discussions and comments.  

\bibliographystyle{amsalpha}
\bibliography{pAdic.bib}

\end{document}